\documentclass[12pt]{amsart}
\usepackage[active]{srcltx}
\usepackage{a4wide}
\usepackage{amsthm,amsfonts,amsmath,mathrsfs,amssymb}
\usepackage{dsfont}
\usepackage[T1]{fontenc}
\usepackage[utf8]{inputenc}\RequirePackage{amsmath}
\usepackage{fourier}\RequirePackage{amssymb}
\usepackage{soul}
\usepackage{pdfpages}\RequirePackage{epsfig}
\usepackage{graphicx}
\def\a{\alpha}       \def\b{\beta}        
\def\d{\delta}       \def\la{\lambda}     
              
                  \def\z{\zeta}
\def\ch{\chi}         

\def\e{\varepsilon}       

\def\G{\Gamma}           
\def\F{\Phi}

\def\D{{\mathbb D}}  \def\T{{\mathbb T}}
\def\C{{\mathbb C}}  \def\N{{\mathbb N}}
\def\R{{\mathbb R}}

\def\ch{{\mathcal H}}    
    \def\cp{{\mathcal P}}
\def\rad{{\mathcal R}}   

\def\({\left(}       \def\){\right)}

\newtheorem{conj}{\sc Conjecture}
\newtheorem{prop}{\sc Proposition}
\newtheorem{lem}[prop]{\sc Lemma}
\newtheorem{thm}[prop]{\sc Theorem}
\newtheorem{cor}[prop]{\sc Corollary}

\newtheorem{other}{\sc Theorem}              

\begin{document}
\title[On the generalized Zalcman functional]{Applications of Livingston-type inequalities to the generalized Zalcman functional}
\author[I. Efraimidis]{Iason Efraimidis}
\address{Departamento de Matem\'aticas, Universidad Aut\'onoma de
Madrid, 28049 Madrid, Spain}
\email{iason.efraimidis@uam.es}
\author[D. Vukoti\'c]{Dragan Vukoti\'c}
\address{Departamento de Matem\'aticas, Universidad Aut\'onoma de
Madrid, 28049 Madrid, Spain} \email{dragan.vukotic@uam.es}
 \urladdr{http://www.uam.es/dragan.vukotic}
\thanks{The authors are partially supported by MTM2015-65792-P from MINECO/FEDER-EU and the Thematic ResearchNetwork MTM2015-69323-REDT, MINECO, Spain.}
\date{17 January, 2017.}
\begin{abstract}
We obtain sharp estimates for a generalized Zalcman coefficient functional with a complex parameter for the Hurwitz class and the Noshiro-Warschawski class of univalent functions as well as for the closed convex hulls of the convex and starlike functions by using an inequality from \cite{E}. In particular, we generalize an inequality proved by Ma for starlike functions and answer a question from his paper \cite{M2}. Finally, we prove an asymptotic version of the generalized Zalcman conjecture for univalent functions and discuss various related or equivalent statements which may shed further light on the problem.
\end{abstract}
\maketitle
\section{Introduction} \label{sect-intro}
\par
Let $\D$ denote the unit disk in the complex plane and $S$ the class of normalized univalent (analytic and one-to-one) functions in $\D$ with the Taylor series $f(z)=z+\sum_{n=2}^\infty a_n z^n$. An instrumental problem in the development of this area for decades was the celebrated Bieberbach's conjecture: $|a_n|\le n$; see \cite{Du} for the history and a survey of the fundamental techniques. The problem was finally solved by L.~de~Bran\-ges in 1984 (see \cite{dB} or \cite{H} for a proof).
\par
It is well known that the coefficients of any $f$ in $S$ satisfy $|a_2^2-a_3| \le 1$; \textit{cf.\/} \cite[Theorem~1.5 and Problem~2, p.~23]{P}. \textit{Zalcman's conjecture\/} states that every $f$ in $S$ satisfies the more general sharp inequality $|a_n^2-a_{2n-1}| \le (n-1)^2$. The importance of the conjecture stems from the fact that it implies the Bieberbach conjecture. This was observed by Zalcman himself in the early 1970s (unpublished); Brown and Tsao \cite{BT} gave a slick short proof. They also proved the conjecture for the typically real and starlike functions \cite{BT}. Ma \cite{M1} did it for the closed convex hull of close-to-convex functions while Krushkal \cite{K1, K2} proved it in the general case for small values of $n$. More general versions of Zalcman's conjecture have also been considered  \cite{BT, M2, LP, LPQ} for the functionals such as $\F (f) = \la\,a_n^2 - a_{2n-1}$ and $\F (f) = \la  a_m a_n-a_{m+n-1}$ for certain positive values of $\la$. These functionals are important because they appear frequently in the coefficient formulas for the inversion transformation in the theory of univalent functions \cite[Ch.~2, p.~28]{Du}.
\par
In this paper we consider the general Zalcman-type functionals with complex $\la$, thus improving a number of results from our earlier unpublished manuscript \cite{EV}. The methods employed here are new and quite different; the results are much more general and do not seem likely to be obtained by the techniques used in \cite{EV}. As an improvement with respect to the earlier papers, we mention the following points:
\par
$\bullet$ Unlike in the earlier literature, here we formulate sharp results for all possible complex values of the parameter $\la$ instead of just the real and positive ones; this may shed some additional light on the conjecture. Specifically, we prove such results for the small subclasses of the normalized univalent functions $S$ such as the Hurwitz and Noshiro-Warschawski classes and also for the closed convex hulls of convex functions and starlike functions, two classes that also contain non-univalent functions.
\par
$\bullet$ Thanks  to our Lemma~\ref{Yakub}, each result formulated as an inequality that holds for all $\la\in\C$ can also be enunciated in an equivalent way as a single new inequality for the coefficients, which could be of some independent interest.
\par
$\bullet$ Theorem~\ref{thm-h} reflects a new phenomenon for the generalized functional $\F (f)$ $= \la  a_m a_n-a_{m+n-1}$: the sharp bounds obtained differ in an essential way in the case $m\neq n$ from the case $m=n$.
\par
$\bullet$ We generalize the results proved by Brown and Tsao \cite{BT} and by Ma \cite{M2} for starlike functions and also answer a question from  \cite{M2} on the smallest positive $\la$ for which Ma's estimates hold. \par
$\bullet$ We show that the generalized Zalcman conjecture is asymptotically true for every complex value of $\la$ and is also equivalent to other related statements which may provide further insight into the problem.
\par
$\bullet$ We improve upon the observation that the Zalcman conjecture implies the Bieberbach conjecture by showing that this implication passes through three related but weaker conjectures than Zalcman's which may be of independent interest.
\par\smallskip
\textsc{Acknowledgments}. {\small The authors would like to thank Professor Dmitry V. Yakub\-ovich for some useful comments and, in particular, for suggesting that a statement like Lemma~\ref{Yakub} should be true.}
\section{Preliminaries} \label{sect-prelim}
\par\medskip
\textbf{A useful lemma}. The following simple but very useful lemma for  complex numbers will allow us to rephrase several statements (to be proved later) in a different language. In this way, infinitely many conditions can typically be replaced by just a single one of different type.
\begin{lem}\label{Yakub}
Let $a$, $b\in\C$ be arbitrary and $C$, $M>0$. Then
\begin{equation}
 |a+\la b| \le M \max\{C,|\la|\}\,, \quad \text{for all} \ \ \la\in\C
 \label{ineq-max-1}
\end{equation}
if and only if
\begin{equation}
|a|+|b| C\le M C\,.
 \label{ineq-max-2}
\end{equation}
Equality holds in \eqref{ineq-max-1} if and only if it holds in \eqref{ineq-max-2} if and only if $|\la |=C$ and $\arg \la =\arg a - \arg b$.
\end{lem}
\begin{proof}
If \eqref{ineq-max-1} holds, we can choose $\la $ with $|\la |=C$ and $\arg \la =\arg a - \arg b$ to get $|a|+|b| C\le M C$.
\par
Conversely, assuming that $|a|+|b| C\le M C$, the triangle inequality yields
\begin{eqnarray*}
 |a+\la b| &\le & |a|+\left|\frac{\la}{C}\right| |b| C
\\
 &\le & \max\left\{1,\left|\frac{\la}{C}\right|\right\} |a| + \max\left\{1,\left|\frac{\la}{C}\right|\right\} |b| C
\\
 &\le &  M C \max\left\{1,\left|\frac{\la}{C}\right|\right\} = M \max\{C,|\la|\} \,.
\end{eqnarray*}
By inspecting the chains of inequalities in \eqref{ineq-max-1} and \eqref{ineq-max-2}, it is quite direct to see that equality is possible in either case only when $|\la |=C$ and $\arg a = \arg (\la b)$ as claimed.
\end{proof}
\par\medskip
\textbf{On Livingston-type inequalities}. The class $\cp$ of analytic functions $g$ such that Re\,$g>0$ in $\D$, normalized  so that $g(0)=1$, is considered frequently in connection with univalent functions. The classical Carath\'eodory lemma \cite[Chapter~2]{Du} states that the Taylor coefficients $p_n$ of any such function $g$ must satisfy the sharp inequality $|p_n|\le 2$, $n\ge 1$. Another inequality: $|p_n - p_k\,p_{n-k}|\le 2$ for the functions in $\cp$ was proved by Livingston \cite[Lemma~1]{L}. The following generalization was obtained by the first author in \cite{E} and will be crucial in this paper.
\par\medskip
\begin{other} \label{other-ie}
If $g\in\cp$, $g(z)=1+ \sum_{n=1}^\infty p_n z^n$ and $1\leq k \leq n-1$ then
\begin{equation} \label{Efr}
 |p_n - w p_k p_{n-k}| \leq 2\max\{1, |1-2w|\}, \quad \text{for all} \quad w\in\C\,,
\end{equation}
and the inequality is sharp.
\end{other}
\par\medskip
We can now deduce an inequality which may be of independent interest. It can be seen as a generalization of the well-known estimate:
$$
\left|p_2 -\frac{1}{2}p_1^2\right| + \frac{1}{2}|p_1|^2  \leq 2\,,
$$
which can be found in \cite[p.~166]{P} and can also be deduced from the classical Schwarz-Pick lemma.
\par
\begin{prop} \label{prop-known-gen}
If $g\in\cp$, $g(z)=1+ \sum_{n=1}^\infty p_n z^n$ and $1\leq k \leq n-1$ then
\begin{equation} \label{no-w}
\left|p_n -\frac{1}{2}p_k p_{n-k}\right| + \frac{1}{2}|p_k p_{n-k}|  \leq 2\,.
\end{equation}
The inequality is sharp.
\end{prop}
\begin{proof}
Rewriting inequality \eqref{Efr} from Theorem~\ref{other-ie} in the form
$$
 |2p_n -p_k p_{n-k} +(1-2 w) p_k p_{n-k}| \leq 4 \max\{1, |1-2w|\}\,,
$$
the statement follows by Lemma~\ref{Yakub}.
\end{proof}
\par\medskip
We note that the inequality stated in Proposition~\ref{prop-known-gen} also appeared (with a different proof) in Campschroer's thesis \cite[\S1.4]{C}  which contains various interesting ideas on extremal problems.
\par
The combined use of Theorem~\ref{other-ie}, Lemma~\ref{Yakub}, and Proposition~\ref{prop-known-gen} will be the key to a number of results throughout this paper.
\section{Sharp estimates for some special classes} \label{sect-special}
\par
In this section we will obtain various estimates on the generalized Zalcman functional $\Phi(f)=\la  a_m a_n -a_{m+n-1}$ with complex values $\la$. We do this for four different classes of functions which are either subclasses of $S$ or closed convex hulls of important subclasses of $S$ (which also contain non-univalent functions). All estimates are sharp and each one of them is also formulated in an equivalent way
\par\medskip
\textbf{The Hurwitz class}. The name \textit{Hurwitz class\,} is often used to denote the set $\ch$ of all functions $f$ of the form
$$
 f(z)=z+a_2 z^2+ a_3 z^3+\ldots\,,
$$
analytic in $\D$ and with the property that $\sum_{n=2}^\infty n |a_n|\le 1$. Obviously, the $n$-th coefficient of a function in $\ch$ is subject to the estimate $|a_n|\le 1/n$ for each $n$. The simplest example of a function in $\ch$ is the polynomial $P_n (z)= z + \frac{z^n}{n}$, $n\ge 2$. It is a well-known exercise that $\ch\subset S$. The reader is referred to \cite{G} for further properties of $\ch$.
\par
For the functions in this class we obtain a much smaller bound on the Zalcman functional than for the entire class $S$. We stress the difference between items (a) and (b) of the theorem below: the estimates on the functional $\Phi(f)=\la  a_m a_n -a_{m+n-1}$ differ in an essential way in the cases $m=n$ and $m\neq n$, with the presence of an extra factor of four in the denominator in the latter case.
\par\medskip
\begin{thm} \label{thm-h}
(a) If  $f\in \ch$ and $n \ge 2$ then the following inequality holds for the coefficients of $f$:
\begin{equation} \label{h-ns}
 n^2|a_n^2| + (2n-1) |a_{2n-1}| \le 1\,.
\end{equation}
This single inequality is equivalent to
\begin{equation} \label{h-n}
 |\la\,a_n^2 - a_{2n-1}| \leq \max\left\{ \frac{|\la|}{n^2},
 \frac{1}{2n-1}\right\}\,, \quad \textrm{for\, all\, } \la\in\C\,.
\end{equation}
Equality holds if and only if
$$
 f(z) =
\begin{cases}
 z + \frac{\alpha}{2n-1} \, z^{2n-1}, \quad \text{for} \; \; |\la|\leq\frac{n^2}{2n-1}  \\[\jot]
 z + \frac{\alpha}{n} \, z^n, \quad\qquad \text{for} \; \; |\la| \geq \frac{n^2}{2n-1},
\end{cases}
$$
where $\alpha$ is a complex number of modulus one.
\par\smallskip
(b) If $f\in \ch$, then for any two distinct values $m$, $n\ge 2$ we have
\begin{equation} \label{h-mns}
 4 m n |a_m a_n| + (m+n-1) |a_{m+n-1}| \le 1\,.
\end{equation}
The last inequality is equivalent to
\begin{equation} \label{h-mn}
 |\la\,a_m a_n - a_{m+n-1}| \leq \max\left\{ \frac{|\la|}{4mn},  \frac{1}{m+n-1}\right\}\,, \quad \textrm{for\, all\, } \la\in\C\,.
\end{equation}
Equality holds if and only if
$$
f(z) =
\begin{cases}
z + \frac{\alpha}{m+n-1} \, z^{m+n-1}, \quad \text{for} \; \; |\la|\leq\frac{4mn}{m+n-1}  \\[\jot]
z + \frac{\alpha}{2m} \, z^m +\frac{\beta}{2n} \, z^n , \;\quad \text{for} \; \; |\la| \geq \frac{4mn}{m+n-1},
\end{cases}
$$
where $\alpha$ and $\beta$ are complex numbers such that $|\alpha|=|\beta|=1$.
\end{thm}
\begin{proof}
(a) By the definition of $\ch$ we have that $n|a_n| \le1$ and  therefore
$$
 n^2|a_n|^2+(2n-1)|a_{2n-1}| \; \le \; n|a_n| +(2n-1)|a_{2n-1}| \; \le \; 1.
$$
Taking
$$
 M=\frac{1}{n^2}\,, \quad C=\frac{n^2}{2n-1}
$$
in Lemma \ref{Yakub}, the above inequality is equivalent to \eqref{h-n}. Obviously, equality is only possible when $n|a_n|=1$ or $n|a_n|=0$. The first case implies that $a_{2n-1}=0$ and all remaining coefficients are zero. The second yields that $(2n-1)a_{2n-1}=1$ and all remaining coefficients are zero, which easily leads to the desired conclusion.
\par\smallskip
(b) The proof is slightly more involved in the case $m\neq n$. Set $x=m|a_m|$ and $y=n|a_n|$. Clearly $x,y\ge 0$ and by the definition of $\ch$ they satisfy $x+y\le1$. This and  $(x-y)^2\ge 0$ imply
$$
 4xy \le (x+y)^2 \le \, x+y\,.
$$
It follows readily from the definition of $\ch$ that
$$
 4 m n |a_m a_n| + (m+n-1) |a_{m+n-1}| \le 1\,.
$$
Using Lemma \ref{Yakub} with
$$
 M=\frac{1}{4mn}\,, \quad C=\frac{4mn}{m+n-1}
$$
we see that this is equivalent to  \eqref{h-mn}.
\par
Equality holds in (b) if and only if either $m|a_m|=n|a_n|=0$ or $m|a_m|=n|a_n|=1/2$, which again easily leads to the claim on extremal functions.
\end{proof}
\par\medskip
\textbf{The Noshiro-Warschawski class}. We now consider the functions in the normalized class
$$
 \rad = \{ f \in \ch(\D)\,\colon\,\mathrm{Re} f^\prime (z) > 0, f(0) = 0\,, f^\prime (0)=1\}\,.
$$
A typical example of a function in $\rad$ is $f(z)=2\log\dfrac1{1-z} - z$ whose derivative is $f^\prime(z)= (1+z)/(1-z)$, a mapping of $\D$ onto the right half-plane. The branch of the logarithm is chosen so that $\log 1=0$.
\par
Note that $\rad\subset S$ by the basic Noshiro-Warschawski lemma \cite[Theorem~2.16]{Du}. MacGregor \cite{MG} showed that for $f$ in $\rad$ we have $|a_n|\le 2/n$. On the other hand, $\rad$ contains the Hurwitz class $\ch$. This can be seen as follows. If $f$ is a function in $\ch$ other than the identity, then  $f^\prime(0)=1$ and, when $z\neq 0$, we have the strict inequality
$$
 \textrm{Re}\,f^\prime (z) = 1 + \sum_{n=2}^\infty n \textrm{Re}\,\{a_n z^{n-1}\} \ge 1 - \sum_{n=2}^\infty n |a_n| |z|^{n-1} > 1 - \sum_{n=2}^\infty n |a_n| \ge 0\,.
$$
It should not be too surprising to have larger upper bounds for the generalized Zalcman functional among the functions in $\rad$ than for those in $\ch$. This is indeed the case, as our next result shows.
\par
\begin{thm} \label{thm-nw}
Let $f\in\rad$ and $m,n\geq2$. Then the following inequality holds for the coefficients of $f$:
$$
\left| \frac{mn}{2(m+n-1)} a_m a_n - a_{m+n-1}\right| + \frac{mn |a_m a_n|}{2(m+n-1)} \leq \frac{2}{m+n-1}\,.
$$
This is equivalent to
$$
|\la  a_m a_n -a_{m+n-1}| \leq \frac{2}{m+n-1} \max\left\{1, \left|1-2\la \frac{m+n-1}{mn}\right|\right\} \quad \text{for all} \quad \la \in\C\,.
$$
Equality holds in both inequalities for the function
\begin{equation}
 f(z) =2 \log\frac{1}{1-z} -z
 \label{R}
\end{equation}
when $\left|1-2\la \frac{m+n-1}{mn}\right|\geq1$ and for
$$
f(z) = \int_{[0,z]} \frac{1+\z^{m+n-2}}{1-\z^{m+n-2}} d\z
$$
(meaning integration over the segment from $0$ to $z$) when $\left|1-2\la \frac{m+n-1}{mn}\right|<1$.
\end{thm}
\begin{proof}
Let $f\in\rad$, $f(z)=z+\sum_{n=2}^\infty a_n z^n$ in $\D$. Then $g=f^\prime\in \cp$ and, writing $g(z)=1+\sum_{n=1}^\infty p_n z^n$, the coefficients of $f$ and $g$ are related by $p_{n-1}=n a_n$. The desired inequalities now follow from Theorem~\ref{other-ie} and Proposition~\ref{prop-known-gen}.
\par
The function given by \eqref{R} has coefficients $2/n$ and yields equality in the cases indicated. For the remaining case, when $\left|1-2\la \frac{m+n-1}{mn}\right|<1$, we find that the function
$$
f^\prime(z) = \frac{1+z^{m+n-2}}{1-z^{m+n-2}}
$$
belongs to $\cp$. Since $f(0)=0$, it follows that $f\in\rad$. Clearly, $$
f(z) = z+ \sum_{k=1}^\infty \frac{2}{k(m+n-2)+1} z^{k(m+n-2)+1}\,,
$$
and it is easily checked that equality is attained for this function.
\end{proof}
We observe that one can write down an explicit formula for the extremal function written above as a primitive but there is really no need for this.
\par\medskip
\textbf{The closed convex hull of convex functions}. Denote by $C$ the class of convex functions in $S$. A typical example is the half-plane function $\ell(z)=\frac{z}{1-z}$. It is well known that the coefficient estimate can be improved a great deal for the functions in $C$: by a theorem of Loewner, they must satisfy $|a_n|\le 1$, with equality only for the function $\ell$ and its rotations (see \cite[Corollary on p.~45]{Du}).
\par
Denote by ${\rm co}(C)$ the convex hull of $C$ and by $\overline{{\rm co}(C)}$ its closure in the topology of uniform convergence on compact subsets of $\D$. Note that this larger class no longer consists exclusively of univalent functions. A well-known result of Marx and Strohh\"acker \cite[p.45]{P} implies that
$$
 \overline{{\rm co}(C)} = \big\{ f\in H(\D) \, : \, {\rm Re} \big(f(z)/z\big) > 1/2, \:  f(0)=f^\prime (0)-1=0 \big\}\,.
$$
Thus, a connection with the class $\cp$ is readily established by the formula
$$
 g(z)=
\begin{cases}
 2\frac{f(z)}{z}-1, \quad \text{for} \; \; z\neq 0 \\[\jot]
 1, \quad\qquad \text{for} \; \; z=0,
\end{cases}
$$
that is, $f\in \overline{{\rm co}(C)}$ if and only if $g\in \cp$.
\par\medskip
For real parameters $\la$ and in the case when $m=n$, the inequality in the following theorem appeared in \cite{EV} for $0\le \la \le 2$ and in \cite{LPQ} for $\la\ge2$. Here we give a complete answer for all complex $\la$ and all $m,n\geq2$.
\begin{thm} \label{thm-conv}
Let $f$ be in $\overline{{\rm co}(C)}$ and $m,n\geq2$. Then
$$
 |a_m a_n -a_{m+n-1}| +|a_m a_n|  \le 1\,.
$$
This is equivalent to the following statement:
$$
 |\la a_m a_n -a_{m+n-1}| \leq \max\{1, |1-\la |\}\,, \quad \text{for all} \quad \la\in\C\,.
$$
Equality holds in both inequalities for the function given by
\begin{equation}\label{C}
 f(z) =\frac{z}{1-z}
\end{equation}
when $|1-\la|\geq1$ and for
\begin{equation}\label{C-extr}
 f(z) =\frac{z}{1-z^{m+n-2}}
\end{equation}
when $|1-\la |<1$.
\end{thm}
\begin{proof}
The function $g$ given by
$$
 g(z)=2\frac{f(z)}{z}-1=1+\sum_{n=1}^\infty p_n z^n\,, \quad z\neq 0\,, \qquad g(0)=1\,,
$$ belongs to $\cp$ and the coefficients of the functions $f$ and $g$ are related by $p_{n-1}=2a_n$. Theorem~\ref{other-ie} yields the desired inequality in $\la$ and the equivalent formulation as a single inequality follows by Lemma~\ref{Yakub}.
\par
The function given by \eqref{C} clearly yields equality in the cases indicated. For the remaining case, when $|1-\la |<1$, we consider the function
$$
g(z) = \frac{1+z^{m+n-2}}{1-z^{m+n-2}}\,,
$$
which belongs to $P$. Let $f$ be the function in $\overline{{\rm co}(C)}$ for which $g(z)=2f(z)/z-1$. We see that
$$
 f(z)=\frac{z}{1-z^{m+n-2}}=\sum_{k=0}^\infty z^{k(m+n-2)+1}
$$
and equality is attained for this function.
\end{proof}
\par\medskip
\textbf{A further generalization}. More generally, one can consider the class $C(\a)$ of analytic functions in $\D$ of the form $f(z)=z+a_2 z^2+ a_3 z^3+\ldots$ which satisfy
$$
 {\rm Re} \left( 1+ \frac{zf^{\prime\prime}}{f^\prime}\right) > \a.
$$
This class was introduced by Robertson in \cite[Section~3]{Ro36}, \textit{cf.} also \cite[\S 2.3]{GK}. Of course, $C(0) = C$, the class of convex functions. Since these classes become smaller as $\a$ increases, all functions in $C(\a)$ are univalent and convex whenever $0\le\a<1$. When $-1/2\le\a<0$ these functions are known to be univalent and convex in one direction \cite{U52}.
\par
The function given by
$$
 f_\a(z)=
\begin{cases}
 \frac{1-(1-z)^{2\a-1}}{2\a-1}, \quad \text{for} \; \; \a\neq1/2, \\[\jot]
 \log\frac{1}{1-z}, \quad\quad \text{for} \; \; \a=1/2,
\end{cases}
$$
is often extremal in this class. Its coefficients are easily computed: $$
 A_n = \frac{\G(n+1-2\a)}{n! \, \G(2-2\a)} \, = \, \frac{1}{n!}\prod_{k=2}^n(k-2\a).
$$
It is known that $|a_n|\leq A_n$ for functions in $C(\a)$ (see \cite{Ro36} for $0\le \a<1$ and \cite{Su05} for starlike functions of any order $\a<1$, which directly implies the inequality considered here through Alexander's Theorem).
\par
Arguments similar to those used earlier allow us to recover, without much effort, a recent theorem from \cite{LPQ}; thus, we omit several details below.
\par\medskip
\begin{thm} Let $\a<1$, $f \in \overline{{\rm co}(C(\a))}$,  $m,n\ge2$, and $A_n$ as above. Then
$$
 \left|\frac{a_m a_n}{A_m A_n} - \frac{a_{m+n-1}}{A_{m+n-1}}\right| + \frac{|a_m a_n|}{A_m A_n} \le 1\,.
$$
This is equivalent to the following statement:
$$
 |\la a_m a_n -a_{m+n-1}| \leq \max\{A_{m+n-1},|\la A_m A_n -A_{m+n-1}|\}\,, \quad \text{for all} \quad \la\in\C\,.
$$
Equality holds in both inequalities above for the function given by $f = f_\a$ in the case when $|\la A_m A_n -A_{m+n-1}| \ge A_{m+n-1}$ and for the function
$$
f(z) = \frac{1}{m+n-2}\sum_{k=1}^{m+n-2} e^{-\frac{2\pi k i}{m+n-2}} f_\a\left(e^\frac{2\pi k i}{m+n-2} z\right)
$$
in the case when $|\la A_m A_n -A_{m+n-1}| < A_{m+n-1}$.
\end{thm}
\begin{proof}
Theorem~4 in \cite{BHMW} provides the following Herglotz-type representation: there exists a probability measure $\mu$ on $\T$ such that
$$
f^\prime (z) = \int_\T \frac{d\mu(\la)}{(1-\la z)^{2-2\a}}
$$
for every $f \in \overline{{\rm co}(C(\a))}$. From here the coefficients of such $f$ relate with those of a function in $\cp$ by
$$
 a_n = \frac{A_n}{2} p_{n-1}.
$$
The desired result now follows from Theorem~\ref{other-ie} and Proposition~\ref{prop-known-gen} as before.
\end{proof}
\par\medskip
\textbf{The closed convex hull of starlike functions.} A set $E$ is said to be \textit{starlike\/} with respect to the origin if for every $z\in E$ the entire segment $[0,z]$ is contained in $E$. A function $f$ is said to be starlike if it is a univalent function of the disk onto a domain starlike with respect to the origin. The usual notation for the subclass of $S$ consisting of all starlike functions is $S^*$.  Obviously, $C\subset S^*\subset S$.
\par
Brown and Tsao \cite[Theorem~2]{BT} showed that the Zalcman conjecture is true for starlike functions and Ma \cite[Theorem~2.3]{M2}  generalized their result further to show that
$$
 |\la a_m a_n -a_{m+n-1}| \le \la\,m n -m -n +1
$$
whenever $\la\in\R$ and $\la\ge \la_0= \frac{2(m+n-1)}{mn}$. The following result generalizes his result to the case of complex parameters and at the same time answers in the affirmative his question posed in \cite{M2} as to whether $\la_0$ is the smallest positive number for which the above bound remains true.
\par\smallskip
\begin{thm} \label{thm-starl}
Let $f \in \overline{{\rm co}(S^*)}$ and $m,n\geq2$. Then
$$
 \left|\frac{a_m a_n}{mn} - \frac{a_{m+n-1}}{m+n-1}\right| + \frac{|a_m a_n|}{mn} \le 1\,.
$$
This statement is equivalent to
$$
 |\la a_m a_n -a_{m+n-1}| \leq (m+n-1) \max\left\{1, \left|1-\frac{mn}{m+n-1}\la\right|\right\}\,, \quad \mbox{for all} \ \  \la\in\C\,.
$$
In both cases, equality holds for the function given by
\begin{equation}\label{S}
 f(z) =\frac{z}{(1-z)^2}
\end{equation}
when $\left|1-\frac{mn}{m+n-1}\la\right|\geq1$ and for
\begin{equation}\label{S-extr}
f(z) =\frac{z}{1-z^{m+n-2}} + (m+n-2)\frac{z^{m+n-1}}{(1-z^{m+n-2})^2}
\end{equation}
when $\left|1-\frac{mn}{m+n-1}\la\right|<1$.
\end{thm}
\begin{proof}
By Alexander's Theorem \cite[\S2.5]{Du} we know that every starlike function $f$ is of the form $f(z)= z h^\prime(z)$ for some $h$ in $C$. Such a relation is preserved upon taking convex combinations and uniform limits on compact subsets of the disk, hence we obtain the same conclusion for every function $f$ in $\overline{{\rm co}(S^*)}$ and some corresponding $h$ in $\overline{{\rm co}(C)}$. Next, taking into account the connection between the classes $\overline{{\rm co}(C)}$ and $\cp$ we readily get that for every $f$ in $\overline{{\rm co}(S^*)}$ there is a function $g$ in $\cp$ such that
$$
 f(z) = zh^\prime(z) = z \left(\frac{z+zg(z)}{2}\right)^\prime = \frac{z}{2} \big(1+g(z)+zg^\prime(z)\big)\,,
$$
Writing $g(z)=1+ \sum_{n=1}^\infty p_n z^n$ we can easily deduce that $a_n =~\frac{n p_{n-1}}{2}, \, n\geq 2$. Now, Theorem~\ref{other-ie} yields the first and Proposition~\ref{prop-known-gen} the second of the two inequalities.
\par
We note that the Koebe function \eqref{S} clearly satisfies the equality in the cases indicated. For the remaining case, when $\left|1-\frac{mn}{m+n-1}\la\right|<1$, we consider the function
$$
g(z) = \frac{1+z^{m+n-2}}{1-z^{m+n-2}}\,,
$$
which belongs to $\cp$. Hence, in view of the above computation, the function $f=\frac{z}{2}(1+g+zg')$ belongs to $\overline{{\rm co}(S^*)}$ and has the form \eqref{S-extr}. We now compute
$$
f(z) = z+ \sum_{k=1}^\infty \big(k(m+n-2)+1\big) z^{k(m+n-2)+1}\,.
$$
Clearly, equality is attained for this function in both inequalities.
\end{proof}
\par\medskip
The class $\overline{{\rm co}(S^*)}$ is obviously strictly larger than $S^*$ and it turns out that, in the simplest case $\la=1$, the above Theorem~\ref{thm-starl} yields the sharp bound
$$
  |a_m a_n -a_{m+n-1}| \leq \max \{m+n-1,(m-1)(n-1)\}\,,
$$
which is different from $(m-1)(n-1)$ when either $m=2$ or $m=n=3$; this is explained in \cite{M2}. In particular, when $m=n\in\{2,3\}$ we have the estimate $|a_n^2-a_{2n-1}|\le 2n-1$. In this case, $2n-1>(n-1)^2$, the general estimate in the Zalcman conjecture (also confirmed by Brown and Tsao for starlike functions). However, there is no contradiction since the class $\overline{{\rm co}(S^*)}$ also contains non-univalent functions.
\par\medskip
\section{Some general considerations} \label{sect-general}
\par
\textbf{An asymptotic version of the Zalcman conjecture}. Let $f\in S$ and $M_\infty (r,f)=\max_{|z|=r} |f(z)|$. Recall that the \textit{Hayman index\/} of $f$ is the number
$$
 \a=\lim_{r\to 1} (1-r)^2 M_\infty (r,f)\,.
$$
It is well known \cite[p.~157]{Du} that $0\le\a\le 1$. Moreover, Hayman's regularity theorem \cite[Theorem~5.6]{Du} asserts that for each $f$ in $S$ its $n$-th Taylor coefficient $a_n$ satisfies $\lim_{n\to\infty} |a_n/n|=\a$.
\par
Even though the Zalcman conjecture continues to be open problem, we now show that its asymptotic version is true and we give it in a precise quantitative form.
\begin{thm} \label{thm-haym}
Let $f(z)=z+a_2 z^2+\ldots$ be in $S$, with Hayman index $\a$, and let $\la \in\C$. Then
\begin{equation}
 \lim_{m, n\rightarrow\infty} \frac{|\la  a_m a_n-a_{m+n-1}|}{|\la  mn -m-n+1|} = \alpha^2\,.
\label{haym}
\end{equation}
Also, if we define $B_{m,n} (\la ) = \sup_{f\in S}|\la  a_m a_n-a_{m+n-1}|$, then
$$
\lim_{m, n\rightarrow\infty} \frac{B_{m,n}(\la )}{|\la  mn -m-n+1|} = 1.
$$
In both limits, we understand that $(m, n)\rightarrow (\infty,\infty)$ unconditionally in $\N^2$ (meaning that $m+n\to\infty$).
\end{thm}
\begin{proof}
Applying the triangle inequality we get
$$
\frac{|\la  a_m a_n-a_{m+n-1}|}{|\la  mn -m-n+1|} \leq \frac{|a_m a_n|}{m n} \frac{|\la | m n}{|\la  mn -m-n+1|} + \frac{|a_{m+n-1}|}{m+n-1} \frac{m+n-1}{|\la  mn -m-n+1|},
$$
where the right-hand side converges to $\alpha^2$ in view of Hayman's regularity theorem. Analogously, we can use the triangle inequality to get a lower bound converging to $\alpha^2$. Hence \eqref{haym} follows.
\par
The Koebe function clearly shows that $B_{m,n}(\la ) \geq |\la \,mn -m-n+1|$. Using the customary notation $A_n=\sup_{f\in S} |a_n|$, we have
$$
1 \leq \frac{B_{m,n}(\la )}{|\la \,mn -m-n+1|} \leq \frac{|\la | A_m A_n+A_{m+n-1}}{|\la \,mn -m-n+1|} \to 1\,,
$$
when $(m, n)\rightarrow (\infty,\infty)$.
\end{proof}
\par
As is usual \cite[Chapter~2]{Du}, by a \textit{rotation\/} of a function $f$ in $S$ we mean the function $f_c (z)=\overline{c} f(c z)$, $|c|=1$, which is again in $S$. Note that the rotations of the Koebe function give equality in Zalcman's conjecture.
\begin{cor}
If $f\in S$ is not a rotation of the Koebe function, then for every $\d\in (0,1-\alpha^2)$ there exist $m_0$ and $n_0$ in $\mathbb{N}$ (which depend on $f$) such that
$$
|\la  a_m a_n-a_{m+n-1}|\leq (1-\delta)|\la  mn -m-n+1|\,,
$$
for all $m\geq m_0$, $n\geq n_0$.
\end{cor}
\par\medskip
\textbf{Some equivalent reformulations of the Zalcman conjecture}. For the sake of simplicity, we treat only the original conjecture: $|a_n^2-a_{2n-1}| \le (n-1)^2$. We first recall that, if assumed true for all $n$, it easily implies the Bieberbach conjecture (now de Branges' theorem). Since the proof of this implication for one value of $n$ uses the validity of the conjecture for another $n$, in order to avoid this discussion in the sequel, we shall simply take for granted the Bieberbach conjecture for odd integers: $|a_{2n-1}|\le 2n-1$. With this in mind, the Zalcman conjecture can be reformulated in several ways.
\par\smallskip
\begin{thm} \label{thm-eq}
Let $f\in S$ be fixed, $f(z)=z+\sum_{k=2}^\infty a_k z^k$, and let $n\ge 2$ be arbitrary. Then the following statements are equivalent:
\begin{itemize}
\item[(a)] The Zalcman conjecture holds: $|a_n^2-a_{2n-1}| \le (n-1)^2 = n^2-(2n-1)$;
\item[(b)] $|a_n^2-ta_{2n-1}| \le n^2-t(2n-1)$ \ for all $\,t\in [0,1]$;
\item[(c)] $|a_n^2-a_{2n-1}| + r |a_{2n-1}| \le (n-1)^2+r(2n-1)$ \ for all $\,r>0$;
\item[(d)] $|a_n^2-w a_{2n-1}| \le (n-1)^2+|w-1|(2n-1)$ \ for all $\,w\in\C$.
\end{itemize}
\end{thm}
\begin{proof}
We will show that (b) $\Rightarrow$ (a) $\Rightarrow$ (c) $\Rightarrow$ (d) $\Rightarrow$ (b). Of course, other schemes of proof are also possible.
\par\medskip
\fbox{(b) $\Rightarrow$ (a)}\,. This implication is trivial. \par\medskip
\fbox{(a) $\Rightarrow$ (c)}\,. Suppose that (a) holds. In view of the inequality $|a_{2n-1}|\le 2n-1$, we deduce directly from (a) that
$$
 |a_n^2-a_{2n-1}| + r |a_{2n-1}| \le (n-1)^2+r(2n-1)
$$
for all $r>0$.
\par\medskip
\fbox{(c) $\Rightarrow$ (d)}\,. Suppose
$$
 |a_n^2-a_{2n-1}| + r |a_{2n-1}| \le (n-1)^2+r(2n-1)
$$
holds for all $\,r>0$ (hence, by taking limits, also for $r=0$). Let $w$ be arbitrary. If $w=1$ then (d) follows from the assumption for $r=0$. For every other value of $w$ there is a positive $r$ such that $|w-1|=r$ and we get
\begin{eqnarray*}
 |a_n^2-w a_{2n-1}| &=& |a_n^2-a_{2n-1} + (1-w) a_{2n-1}|
\\
 &\le & |a_n^2-a_{2n-1}|+ r |a_{2n-1}|
\\
 &\le & (n-1)^2+r(2n-1)
\\
 &=& (n-1)^2+ |w-1|(2n-1)\,,
\end{eqnarray*}
and (d) is proved.
\par\medskip
\fbox{(d) $\Rightarrow$ (b)}\,. This follows readily by taking $w=t\in [0,1]$.
\end{proof}
\par\medskip
Several remarks are in order to show that Theorem~\ref{thm-eq} may shed some new light on the problem.
\par\smallskip
$\bullet$ In view of Theorem~\ref{thm-eq}, proving the Zalcman conjecture amounts to proving any of the equivalent statements while disproving it would amount to finding one single example of a function which does not satisfy one of the inequalities (b), (c), or (d) for one single value of $t$, $r$, or $w$ respectively.
\par\smallskip
$\bullet$
Statement (b) in the theorem had already been verified for the typically real functions and follows from \cite[Theorem~1]{BT}.
\par\smallskip
$\bullet$ The fact that Bieberbach's conjecture is true means that (d) holds for $w=0$. If Zalcman's conjecture were to be true, we would have many more new inequalities such as, for example,
$$
 |a_n^2-2a_{2n-1}|\le n^2\,,
$$
obtained by taking $w=2$ in (d).
\par\smallskip
$\bullet$ We also note that the validity of Bieberbach's conjecture readily implies that (d) is true for any $w=-M$, where $M$ is real and positive; indeed:
$$
 |a_n^2 + M a_{2n-1}|\le n^2 + M (2n-1) = (n-1)^2 + (M+1)(2n-1)\,.
$$
However, we do not know whether (d) is true in general for \textit{any other\/} value of $w$ except for those in $(-\infty,0]$. So there appears to be a significant gap between Bieberbach and Zalcman.
\par\medskip
\textbf{Three related but weaker conjectures}. At this point it seems natural to formulate three closely related conjectures. They could be of interest since they are both weaker than Zalcman's but each of them also implies the Bieberbach conjecture.
\par
In relation to condition (b) of our preceding theorem, for a given value $t$ in $[0,1]$ we will denote by $(B_t)$ the following statement: \begin{equation}
 \tag{$B_t$}
 |a_n^2 -ta_{2n-1}| \le n^2 -t(2n-1)\,,
\end{equation}
for all $f\in S$ with $f(z)=z+\sum_{k=2}^\infty a_k z^k$, and all $n\ge 2$. Thus, we can formulate the \emph{first weak\/} version of the Zalcman conjecture as follows.
\par
\begin{conj} \label{conj1}
There exists $t\in(0,1]$ such that $(B_t)$ holds.
\end{conj}
\par\noindent
It is not clear in any obvious way that this statement is true. However, $(B_0)$ is precisely the Bieberbach conjecture and we know it is true. Thus, the set of all $t\in [0,1]$ for which $(B_t)$ holds is non-empty. It is easy to see that this set is closed as the defining condition contains a non-strict inequality. It is also convex; indeed, if $(B_s)$ and $(B_t)$ hold and $\a$, $\b\in[0,1]$ with $\a+\b=1$ then clearly
\begin{align*}
 |a_n^2 - (\a s+\b t) a_{2n-1}| & \le \a |a_n^2 - s a_{2n-1}| + \b  |a_n^2 - t a_{2n-1}|
\\
 & \le \a (n^2 - s (2n-1)) + \b (n^2 - t (2n-1))
\\
 & =  n^2 - (\a s + \b t) (2n-1)\,,
\end{align*}
hence $(B_{\a s+\b t})$ is also true. Thus, it seems natural to consider the quantity $T = \sup \{t\in[0,1] \,\colon\, (B_t) \,\text{is true\,}\}$. With this notation, the Zalcman conjecture claims that $T\ge1$, while the weak Zalcman conjecture only claims that $T> 0$.
\par\medskip
Now consider the situation when condition (c) in Theorem~\ref{thm-eq} holds only for \emph{some} $r>0$. So for a fixed $r>0$ we can consider the statement $(C_r)$:
\begin{equation}
 \tag{$C_r$} |a_n^2-a_{2n-1}| + r |a_{2n-1}| \le (n-1)^2+r(2n-1)\,,
\end{equation}
for all $f\in S$ with $f(z)=z+ \sum_{k=2}^\infty a_k z^k$, and all $n\ge 2$. This clearly gives rise to the \emph{second weak\/} version of the Zalcman conjecture.
\par
\begin{conj} \label{conj2}
There exists $r\in [0,1]$ such that $(C_r)$ holds.
\end{conj}
\par\medskip
It also makes sense to consider a weaker version of condition (d) in Theorem~\ref{thm-eq}. For a fixed $r$, say $r\in [0,1]$, consider
\begin{equation}
 \tag{$D_r$} |a_n^2-wa_{2n-1}| \le (n-1)^2 + |w-1| (2n-1)\,, \quad \textrm{for\, all\,} w \ \, \textrm{with\,} |w-1|=r\,,
\end{equation}
for all $f\in S$ with $f(z)=z+ \sum_{k=2}^\infty a_k z^k$, and all $n\ge 2$. Thus, we have the \emph{third weak\/} version of the Zalcman conjecture.
\par
\begin{conj} \label{conj3}
There exists $r\in [0,1]$ such that $(D_r)$ holds.
\end{conj}
\par\medskip
The following relationship exists between the conjectures mentioned.
\par\medskip
\begin{thm} \label{B_t}
Assume only a weaker statement than the Bieberbach conjecture, for example, Littlewood's theorem \cite[Theorem~2.8]{Du}: $|a_n|< e n$, for all $n\ge 2$. Under these assumptions we have:
\begin{itemize}
\item[(a)] The Zalcman conjecture implies Conjecture~\ref{conj3}.
\item[(b)] Conjecture~\ref{conj3} implies Conjecture~\ref{conj2}.

\item[(c)] Conjecture~\ref{conj2} implies Conjecture~\ref{conj1}  (with $t=1-r$).
\item[(d)] Conjecture~\ref{conj1} implies the Bieberbach conjecture.
\item[(e)] All weak conjectures: Conjecture~\ref{conj1}, Conjecture~\ref{conj2}, and Conjecture~\ref{conj3} are asymptotically true. For example, if $f$ is a function in $S$ with Hayman index $\a$, and $t \in [0,1]$ then
\begin{equation}
 \lim_{n\rightarrow\infty} \frac{|a_n^2 - t a_{2n-1}|}{n^2 -t(2n-1)} = \alpha^2\,.
\label{haym}
\end{equation}
\end{itemize}
\end{thm}
\begin{proof} (a) This implication is trivial.
\par\medskip
(b) If Conjecture~\ref{conj3} is true, then for the corresponding value of $r$ we have
$$
 |a_n^2-w a_{2n-1}| \le (n-1)^2+r(2n-1)
$$
for all $w$ on the circle $\{w\,\colon\,|w-1|=r\}$. If $a_n- a_{2n-1}\neq 0$ and $a_{2n-1}\neq 0$ we can choose a (unique) $w$ on this circle with $\arg w= \arg (a_n^2- a_{2n-1}) - \arg a_{2n-1}$ so as to obtain
$$
 |a_n^2-w a_{2n-1}| = |a_n^2 - a_{2n-1} + (1-w) a_{2n-1}| = |a_n^2 - a_{2n-1}| + r |a_{2n-1}|\,,
$$
and $(C_r)$ follows. If any of the values $a_n^2- a_{2n-1}$, $a_{2n-1}$ is zero, the statement also holds trivially.
\par\medskip
(c) Assume that Conjecture~\ref{conj2} is true. For the corresponding $r\in [0,1]$, consider $t=1-r\in [0,1]$. Then by the triangle inequality
\begin{eqnarray*}
 |a_n^2-ta_{2n-1}| &\le & |a_n^2-a_{2n-1}| + r |a_{2n-1}|
\\
 &\le & (n-1)^2+r(2n-1)
\\
 &\le & n^2-t(2n-1)\,,
\end{eqnarray*}
which proves that Conjecture~\ref{conj1} is true.
\par\medskip
(d) To show that Conjecture~\ref{conj1} implies the Bieberbach inequality, we follow the idea of Brown and Tsao from \cite{BT}. Begin with a weaker bound for the $n$-th coefficient, say $|a_n| \le C n$, for some $C>1$, and then improve on it using condition $(B_t)$. As was mentioned, we can start off from Littlewood's theorem and $C=e$. Note that
$$
 t\le 1 < \frac{n^2}{2n-1}\,, \quad \textrm{for \, all\, } n\ge 2\,.
$$
Hence
\begin{align*}
 |a_n|^2 & \le |a_n^2 -ta_{2n-1}| + t|a_{2n-1}|
\\
 & \le n^2 -t(2n-1) + C t (2n-1)
\\
 & =  n^2 + t (C-1) (2n-1)
\\
 & \le C n^2\,.
\end{align*}
Hence, $|a_n| \le \sqrt{C} \, n$. Iterating this procedure, we obtain $|a_n| \le C^{2^{-k}} n$ for all positive integers $k$, which yields $|a_n| \le n$.
\par\medskip
(e) The proof is quite similar to that of Theorem~\ref{thm-haym} so we omit it.
\end{proof}


\end{document}